\documentclass[12pt]{article}
\usepackage{amsmath}
\usepackage{amsthm}
\usepackage{amssymb}
\usepackage{latexsym}
\usepackage{color}
\usepackage{graphicx}
\usepackage[all,cmtip]{xy}

\usepackage[dvipdfm,
 bookmarks=true,bookmarksnumbered=true,setpagesize=false,%
  colorlinks=true,%
  pdftitle={},%
  pdfsubject={},%
  pdfauthor={},%
  pdfkeywords={TeX; dvipdfmx; hyperref; color;},%
  colorlinks=false]{hyperref}

\newcommand{\cO}{\mathcal{O}}

\newcommand{\cD}{\mathcal{D}}
\newcommand{\cE}{\mathcal{E}}
\newcommand{\cF}{\mathcal{F}}

\newcommand{\cL}{\mathcal{L}}
\newcommand{\cM}{\mathcal{M}}

\newcommand{\Z}{\mathbb{Z}}
\newcommand{\G}{\mathbb{G}}
\newcommand{\C}{\mathbb{C}}
\newcommand{\bbC}{\mathbb{C}}

\newcommand{\PP}{\mathbb P}
\renewcommand{\mod}{\mathrm{mod}}
\newcommand{\lcm}{\mathrm{l.c.m.}}

\DeclareMathOperator{\Spec}{\mathrm{Spec}}

\DeclareMathOperator{\Pic}{Pic}
\DeclareMathOperator{\FM}{FM}
\DeclareMathOperator{\Fr}{Fr}
\DeclareMathOperator{\ch}{ch}
\DeclareMathOperator{\PGL}{PGL}
\newcommand{\mc}{\mathcal}
\newcommand{\ord}{\ensuremath{\operatorname{ord}}}

\newcommand{\Span}[1]{\left<#1\right>}

\newcommand{\End}{\operatorname{\mathrm{End}}}
\newcommand{\Ker}{\operatorname{\mathrm{Ker}}}
\renewcommand{\Im}{\operatorname{\mathrm{Im}}}
\newcommand{\Aut}{\operatorname{\mathrm{Aut}}}
\newcommand{\Alb}{\operatorname{\mathrm{Alb}}}
\newcommand{\Br}{\operatorname{\mathrm{Br}}}

\newcommand{\id}{\mathrm{id}}
\newcommand{\rid}{\mathrm{id}}

\newcommand{\Ext}{\mathop{\mathrm{Ext}}\nolimits}

\newcommand{\RHom}{\mathop{\mathbb R\mathrm{Hom}}\nolimits}
\newcommand{\mcRHom}{\mathop{\mathbb R\mathcal{H}om}\nolimits}

\newcommand{\Lotimes}{\stackrel{\mathbb L}{\otimes}}
\mathchardef\mhyphen="2D

\theoremstyle{plain}
\newtheorem{theorem}{Theorem}[section]
\newtheorem{lemma}[theorem]{Lemma}
\newtheorem{proposition}[theorem]{Proposition}
\newtheorem{corollary}[theorem]{Corollary}

\theoremstyle{definition}
\newtheorem{definition}[theorem]{Definition}

\newtheorem{remark}[theorem]{Remark}

\newtheorem{claim}[theorem]{Claim}

\title{Fourier--Mukai partners of elliptic ruled surfaces over arbitrary characteristic fields}
\author{Hokuto Uehara and Tomonobu Watanabe}
\date{}
\begin{document}

\maketitle
\renewcommand{\thefootnote}{}
\renewcommand{\thefootnote}{\arabic{footnote}}
\begin{abstract}
The first author explicitly describes the set of Fourier--Mukai partners of elliptic ruled surfaces over the complex number field in \cite{Ue17}. In this article, we generalize it over arbitrary characteristic fields. We also obtain a partial evidence of the Popa--Schnell conjecture in the proof.
\end{abstract}

\section{Introduction}
 Let us consider the derived category of coherent sheaves $D^b(X)$ for a smooth projective variety $X$ over an algebraically closed field $k$ of $p:=\ch k\ge 0$. We call a smooth projective variety $Y$ a \emph{Fourier--Mukai partner} of $X$ if there exists an equivalence $D^b(X) \cong D^b(Y)$ as $k$-linear triangulated categories. We let $\FM(X)$ denote the set of isomorphism classes of Fourier--Mukai partners of $X$. It is a fundamental question to describe the set $\FM(X)$ explicitly. 
It is known that $|\FM (C)|= 1$ for any smooth projective curves $C$ (see \cite[Corollary 5.46]{Hu06}). 
On the other hand, smooth projective surfaces $S$ may have non-trivial Fourier--Mukai partners: Namely, $|\FM (S)|\ne 1$ may occur. 
Bridgeland, Maciocia and Kawamata show in \cite{BM01} and \cite{Ka02} that 
if a smooth projective surface $S$ over $\C$ has a non-trivial Fourier--Mukai partner $T$, 
then both are abelian surfaces, K3 surfaces or elliptic surfaces with nonzero Kodaira dimension. There exist several known examples of surfaces $S$ with $|\FM (S)|\ne 1$ 
(\cite{Mu01, Og02, Ue04}). 

In this article, we study the set $\FM(S)$ of elliptic ruled surfaces $S$ defined over $k$.
Here, an elliptic ruled surface means a smooth projective surface with a 
$\PP^1$-bundle structure over an elliptic curve. 
We obtain the following theorem, which is a generalization of the result for $k=\bbC$ in \cite{Ue17} to an arbitrary algebraically closed field $k$. 

\begin{theorem}\label{thm:main}
Let $S$ be an elliptic ruled surface defined over $k$ and $\pi \colon S \to E$ be 
a $\PP^1$-bundle over an elliptic curve $E$. If $|\FM (S)|\ne 1$, then $S$ is of the form
\[
S = \PP (\mc O_E\oplus \mc L)
\]
for some $\mc L \in \Pic^0{E}$ of order $m \ge 5$. Furthermore we have
\begin{equation*}\label{eqn:FM=P}
\FM (S)=\{ \PP ( \mc O_{E}\oplus \mc L^i)   \mid   i\in \Z \mbox{ with } (i,m)=1 \mbox{ and }1\le i<m \}/\cong,
\end{equation*}
and $$\lvert\FM (S)\rvert
=\varphi(m)/\lvert H^{\cL}_{\hat{E}} \rvert.$$ Here, $\varphi$ is the Euler function, and
we define 
\begin{equation}\label{eqn:def_H00}
H^{\cL}_{\hat{E}} := \{ i \in (\Z/m\Z)^* \, | \, \exists \phi \in \Aut_0(E) \, \text{such that} \, \phi^*\cL \cong \cL^{i}\} 
\end{equation}
as a subgroup of  $(\Z/m\Z)^*$. We also have $\lvert H^{\cL}_{\hat{E}} \rvert = 2, 4$ or $6$, depending on the choice of $E$ and $\cL$. 
\end{theorem}

In the case $k=\bbC$, it is known (cf.~\cite[Equation (3.4)]{Ue17}) that $S= \PP (\mc O_E\oplus \mc L)$ is a quotient of $F_0 \times \PP^1$  by a cyclic group action, where $F_0$ is an elliptic curve, and the first author uses this fact to describe the set $\FM(S)$ in \cite{Ue17}. 
On the other hand, in the case $p:= \ch k>0$, 
elliptic ruled surfaces $S=\PP( \mc O_{E}\oplus \mc L)$ with $p \mid m$ 
do not admit a similar construction (see \cite[\S 5.1]{TU22}).
Therefore, we need more general treatment to show Theorem  \ref{thm:main}.

In the proof of Theorem \ref{thm:main}, we obtain some evidence of the Popa--Schnell conjecture in \cite{PS11}, which states that for any Fourier--Mukai partners $X'$ of a given 
smooth projective variety $X$, there exists an equivalence 
$D^b(\Alb (X'))\cong D^b(\Alb (X))$ of derived categories of their albanese varieties. 

%%%
%%%
%%%

\begin{proposition}[=Corollary \ref{corollary:base-elliptic}]\label{prop:main2}
Let $X \to A$ and $X' \to A'$ be $\mathbb{P}^n$-bundles over abelian varieties $A$ and $A'$ for $n=1,2$. If $X$ and $X'$ are Fourier--Mukai partners, then so are $A$ and $A'$. Furthermore, the Popa--Schnell conjecture holds true in this case.
\end{proposition}

The plan of this article is as follows. In \S \ref{sec:relative_moduli}, we explain some results and  notation of relative moduli spaces of stable sheaves on elliptic fibrations, a main tool for the study of Fourier--Mukai partners of elliptic surfaces.
We obtain a characterization of Fourier--Mukai partners of elliptic surfaces with non-zero Kodaira dimensions in Theorem \ref{BMelliptic} for arbitrary $p=\ch k$, which was originally proved by Bridgeland in the case $p=0$.

In \S \ref{sec:elliptic}, we show several results on automorphisms of elliptic curves. 
 
In \S \ref{sec:Pirozhkov}, we first explain  Theorem  \ref{theorem:fibrations} by Pirozhkov, and then we apply it to show 
Proposition \ref{prop:main2}.

Finally, in \S \ref{sec:FM}, we first narrow down the candidates of elliptic ruled surfaces with non-trivial Fourier--Mukai partners by Proposition \ref{prop:main2} and the main result in \cite{TU22}, and then prove Theorem \ref{thm:main}.

This article is a part of the second author's doctoral thesis.
%%%%
\paragraph{Notation and conventions}\label{subsec:notation_convention}
All algebraic varieties $X$ are defined over an algebraically closed field $k$
of characteristic $p\ge 0$. A point $x\in X$ means a closed point unless otherwise specified. 
 
For an elliptic curve $E$, $\Aut_0(E)$ is the group of automorphisms fixing the origin.
 
By an \emph{elliptic surface}, we will
always mean a smooth projective surface $S$ together with a smooth projective curve $C$ and a relatively minimal projective morphism $\pi\colon S\to C$ whose general fiber is an elliptic curve.
An \emph{elliptic ruled surface} means a smooth projective surface with a 
$\PP^1$-bundle structure over an elliptic curve.

For a morphism  $\pi\colon X\to Y$ between algebraic varieties, the symbol $\Aut (X/Y)$ stands for the group of automorphisms of $X$ preserving $\pi$.

\paragraph{Acknowledgements}
We would like to thank Professor Shinnosuke Okawa for invaluable suggestions. 
We also appreciate the referee's insightful comments. 
H.U. is supported by the Grants-in-Aid for Scientific Research (No.~18K03249, 23K03074). 

%%%%%%%%%%%%%%%%%%%%%%%%%%%%%%%%%%%%%%%%%%%%%%%%%%
\section{Relative moduli spaces of sheaves on elliptic fibrations}\label{sec:relative_moduli}
%%%%%%%%%%%%%%%%%%%%%%%%%%%%%%%%%%%%%%%%%%%%%%%%%%

\subsection{Fourier--Mukai partners of elliptic surfaces}\label{subsec:FMeliiptic_surface}
For a smooth projective variety $X$ defined over an algebraically closed field $k$ of characteristic $p\ge 0$, we denote by $D^b(X)$ the bounded derived categories of coherent sheaves on $X$. We call a smooth projective variety $Y$ a \emph{Fourier--Mukai partner} 
of $X$ if $D^b(X)$ is $k$-linear triangulated equivalent to $D^b(Y)$.
We denote by $\FM (X)$ the set of isomorphism classes of Fourier--Mukai partners 
of $X$.

We study the set $\FM(S)$ for elliptic surfaces $S$. 
Let $\pi\colon S\to C$ be an elliptic surface and denote  a general fiber of $\pi$ by $F_\pi$. We define
\begin{equation}\label{eqn:pi}
\lambda_{\pi}:=\min\{ D\cdot F_\pi \mid D\mbox{ is a horizontal effective divisor on }S \}.
\end{equation}
Fix a polarization on $S$ and consider the relative moduli scheme $\cM(S/C)\to C$ of stable purely $1$-dimensional sheaves\footnote{Here we consider the Gieseker stability, equivalently the slope stability for $1$-dimensional sheaves.  Moreover, the stability does not depend on the choice of polarizations for such sheaves.} 
on the fibers $\pi$, whose existence is assured by Simpson in the case $p=0$
 in \cite{Si94}, and by Langer
in the case of arbitrary $p$ in \cite{La04}. For integers $a>0$ and $i$ with $i$ coprime to $a\lambda_{\pi}$, let $J_S (a,i)$ be the union of those components of $\cM (S/C)$ which contains a point representing a rank $a$, degree $i$ vector bundle on a smooth fiber of $\pi$.
Bridgeland shows in \cite{Br98}  that $J_S (a,i)$ is actually a smooth projective surface and 
the natural morphism $J_S (a,i)\to C$ %, taking a point representing a sheaf supported on the fiber $\pi ^{-1}(x)$ of $S$ to the point $x$,  
 is a minimal elliptic fibration.

Put $J^i(S):=J_S(1,i)$.  We can also define an elliptic surface $J^j(S)\to C$ for arbitrary $j\in \Z$, which is not necessarily fine but the coarse moduli space of a suitable functor  (see \cite[\S 11.4]{MR3586372}).
We have $J^0(S)\cong J(S)$, the Jacobian surface associated to $S$, 
$J^1(S)\cong S$ and 
\begin{equation}\label{eqn:group_structure}
J^i(J^j(S))\cong J^{ij}(S)
\end{equation}
for $i,j\in \Z$. See the argument after \eqref{eqn:ixi} for the proof of \eqref{eqn:group_structure}.

It is well-known that the following statement holds in the case $p=0$ by \cite[Theorem 1.2]{Br98}.
We state that it is also true for arbitrary $p$.

\begin{proposition}\label{prop:SJS}
Elliptic surfaces $S$ and $J^i(S)$ for some integer $i$ with $(i,\lambda _{\pi})=1$ 
are derived equivalent via an integral functor 
$\Phi ^{\mc P}:=\Phi ^{\mc P}_{J^i(S)\to S}$ for a universal sheaf $\mc P$ on $J^i(S)\times S$.
\end{proposition}
\begin{proof}
To prove the statement for $p=0$, Bridgeland first 
applies the Bondal--Orlov's criterion \cite{BO95} (see also \cite[Proposition 7.1]{Hu06}) for the functor 
$\Phi ^{\mc P}$ to be fully faithful,
namely he checks the strongly simpleness of $\mc P$. Then it is easy to show  
$\Phi ^{\mc P}$ is an equivalent by checking the 
Bridgeland's criterion \cite{Br99} for $\Phi ^{\mc P}$ to be equivalent.
But the Bondal--Orlov's criterion is false in the case $p>0$ \cite[Remark 1.25]{HLS07}. 
Instead, if we put an extra assumption that 
the Kodaira--Spencer map $\Ext_{J^i(S)}^1(\mc O_x,\mc O_x)\to \Ext_S^1(\mc P_x,\mc P_x)$
is injective, we  see the proof of \cite{BO95} works, and so the criterion holds 
(see also \cite[Step 5 in the proof of Proposition 7.1]{Hu06}). 
Actually, the map is an isomorphism in our case because $\mc P$ is a universal family. 
This completes the proof.
\end{proof}

We have a nice characterization of 
Fourier--Mukai partners of elliptic surfaces with non-zero Kodaira dimensions.

%%%%%%%%%
%%%%%%%%%
%%%%%%%%%

\begin{theorem}\label{BMelliptic}
Let $\pi \colon S\to C$ be an elliptic surface 
and $T$ a smooth projective variety.
Assume that the Kodaira dimension $\kappa (S)$ is non-zero.
Then the following are equivalent.
\begin{enumerate}
\item 
$T$ is a Fourier--Mukai partner of $S$. 
\item 
$T$ is isomorphic to $J^i(S)$ for some integer $i$ with $(i,\lambda _{\pi})=1$. 
\end{enumerate}
\end{theorem}

\begin{proof}
It follows from Proposition \ref{prop:SJS} that  (ii) implies (i).
The opposite direction was proved in \cite[Proposition 4.4]{BM01} 
when $p=0$ and $S$ has no $(-1)$-curves.
The most of the proof there works even for $p>0$.
So we give only a sketch of the proof. 

As the proof in \cite[Proposition 4.4]{BM01}, we can show
that there exists an equivalent functor $\Phi^{\mc U}\colon D^b(T)\to D^b(J^i(S))$
for some integer $i$ with $(i,\lambda _{\pi})=1$
such that $\Phi^{\mc U}(\mc O_t)=\mc O_y$ for some $t\in T,y\in J^i(S)$. 
Then as in \cite[Lemma 2.5]{BM01},
we see that there exists a rational map $f\colon T\dashrightarrow J^i(S)$ such that 
the kernel $\mc U$ is supported on the graph of $f$ near the point $(t,y)$. 
Because $\Phi^{\mc U}$ is an equivalence, 
we can avoid the possibility that $f$ is inseparable, and hence
$f$ is a birational map. 
Then the proof of  \cite[Proposition 4.4]{BM01} works in the rest (including the case
that $S$ is not minimal). 
\end{proof}

As a consequence of Theorem \ref{BMelliptic},
we obtain
$$
\FM(S) = \{J^i(S) \mid i \in \Z, \ (i,\lambda_{\pi})=1 \} / \cong. 
$$
Moreover we  see that there exist natural isomorphisms 
\begin{equation}\label{eqn:JiJ-i}
J^i(S)\cong J^{i+\lambda_{\pi}}(S)\cong J^{-i}(S).
\end{equation}
Hence, in order to count the cardinality of the set 
$
\FM(S),
$
we often regard an integer $i$ as an element of the unit group $(\Z/\lambda_{\pi}\Z)^*$.
%%%
\if0
Suppose that we are given an isomorphism 
$$\varphi\colon S_1\to S_2$$ 
between elliptic surfaces $\pi_l\colon S_l\to C_l$ 
$(l=1,2)$.
Take a universal family $\mc{U}_l$ on $S_l\times_{C_l}
 J^i(S_l)$ for $i$ with $(i,\lambda_{\pi_l})=1$.
Then there exists a unique isomorphism 
$$J^i(\varphi)\colon J^i (S_1)\to J^i(S_2)$$ such that 
$$(\id_{S_1}\times (J^i(\varphi))^{-1})^*\mc{U}_1\cong (\varphi\times \id_{J^i(S_2)})^*\mc{U}_2\otimes p_2^*\mc{L}$$ 
for some $\mc{L}\in \Pic J^i(S_2)$.
%
Thus,

For an elliptic surface $\pi\colon S\to C$
% with $\kappa(S)\ne 0$, 
\fi
%%%
It follows from 
the isomorphisms \eqref{eqn:group_structure} and \eqref{eqn:JiJ-i} that the set
\begin{equation}\label{eqn:Hpi}
H_{\pi}:=\{ i\in (\Z/\lambda_{\pi}\Z)^* \mid J^i(S)\cong S\}
\end{equation}
forms a subgroup of $(\Z/\lambda_{\pi}\Z)^*$. 
Moreover, we see from \eqref{eqn:group_structure} that 
$J^i(S)\cong J^j(S)$ for $i,j\in (\Z/\lambda_{\pi}\Z)^*$ if and only if 
$(S\cong) J^1(S)\cong J^{i^{-1}j}(S)$. Combining all together,  we have the following.
\begin{lemma}\label{lem:FM_group}
For an elliptic surface $\pi\colon S\to C$ with 
$\kappa(S)\ne 0$,
the set $\FM(S)$ is naturally identified with the group
$(\Z/\lambda_{\pi}\Z)^*/H_\pi$.
\end{lemma} 
Since  $H_\pi$ contains the subgroup $\{\pm 1\}$ if $\lambda_\pi\ge 3$, 
we see
\begin{equation}\label{eqn:FMeqn}
|\FM (S)|\le \varphi (\lambda_{\pi})/2,
\end{equation}
where $\varphi$ is the Euler function.

%We need the following lemma afterwards.

%%%
%%%
\begin{lemma}\label{lem:SJS}
Let $\pi \colon S\to C$ be an elliptic surface. Then we have the following.
 \begin{enumerate}
 \item
 For $i\in \Z$ with $(i,\lambda_\pi)=1$, consider the elliptic fibration $\pi_i\colon J^i(S)\to C$.
 The multiplicities of the fibers $F_x$ and $F'_x$ of $\pi$ and $\pi_i$ over a fixed point $x\in C$ coincide. Furthermore, if the fiber $F_x$ is smooth, then it is isomorphic to $F'_x$.
\item
Let $S$ be an elliptic ruled surface, and take $S'\in \FM (S)$. Then $S'$ is also an elliptic ruled surface with an
elliptic fibration.
\end{enumerate}
\end{lemma}

\begin{proof}
(i) The first statement will be explained by  using Weil--Ch\^{a}telet group in \S \ref{subsec:WC}. See the argument around \eqref{eqn:ord_iximx}.
%The first statement can be proved as in the proof of \cite[Lemma 4.3]{BM01}.
By the property of the relative moduli scheme, the fiber $F'_x$  is the fine moduli space of line bundles of degree $i$ on a smooth elliptic curve $F_x$. Consequently, the second statement follows. 

(ii) Theorem \ref{BMelliptic} implies that there exists an integer $i$ with $(i,\lambda_\pi)=1$ such that 
$J^i(S)\cong S'$, which implies that $S'$ has an elliptic fibration $\pi'$. The Kodaira dimension is derived invariant 
by \cite[Corollary 4.4]{To06}, hence $S'$ is a rational elliptic surface or an elliptic ruled surface. 
Then, \cite[Theorem B]{MR3750214} implies that $S'$ is also an elliptic ruled surface.
\end{proof}

\subsection{Weil--Ch\^{a}telet group}\label{subsec:WC}
In this subsection, we recall the definition of  the Weil--Ch\^{a}telet group. For more details, see \cite[Ch.X.3]{Si09} and \cite[Ch.11.5]{MR3586372}. 

Let $E_0$ be an elliptic curve over a field $K$.  
A homogeneous space for $E_0$ is a pair $(E, \mu)$, where $E$ is a smooth curve over $K$, and $\mu$ is a simply transitive algebraic group action
\[
\mu \colon E \times E_0 \to E. 
\]

We say that two homogeneous spaces $(E, \mu)$ and $(E', \mu')$ are \emph{equivalent} if there exists an isomorphism $\theta \colon E \to E'$ defined over $K$ which is compatible with the action of $E_0$. The collection $WC(E_0)$ of equivalence classes of homogeneous spaces for $E_0$ 
 has a natural group structure (cf.~\cite[Theorem X.3.6]{Si09}, \cite[Proposition 11.5.1]{MR3586372}),
 and it is called the \emph{Weil--Ch\^{a}telet group}. 

Let $\pi\colon S\to C$ be an elliptic surface (over an algebraically closed field $k$). 
We denote the generic fiber of $\pi_i\colon J^i(S)\to C$ by $J^i_{\eta}$ for $i\in \Z$. 
Then $J^0_{\eta}$ is an elliptic curve over the function field of $C$, and  we have a natural homogeneous space structure
$$
\mu_i \colon J^i_{\eta} \times J^0_{\eta} \to J^i_{\eta}\quad  (\mc{L},\mc{M})\mapsto \mc{L}\otimes \mc{M},
$$  
and hence we can regard 
$
(J^i_{\eta},\mu_i) \in WC(J^0_{\eta}).
$
We define 
\begin{equation}\label{eqn:xi}
\xi:=(J^1_{\eta},\mu_1)\in WC(J^0_{\eta}),
\end{equation}
then, we have
\begin{equation}\label{eqn:ixi}
i\xi=(J^i_{\eta},\mu_i)
\end{equation}
(cf.~\cite[Remark 11.5.2]{MR3586372}) and thus 
\begin{equation}\label{eqn:order_xi}
\ord \xi \mid \lambda_{\pi}.
\end{equation}
It follows from \eqref{eqn:ixi} that the generic fibers of $J^i(J^j(S))\to C$ and $J^{ij}(S)\to C$ are isomorphic to each other, and taking the relative smooth minimal models of  compactifications of generic fibers, we obtain $J^i(J^j(S))\cong J^{ij}(S)$ as in \eqref{eqn:group_structure}.
  
Take a closed point $x\in C$ and consider the henselization of the local ring $\mc{O}_{C,x}$ and denote it by $\mc{O}_{C,x}^h$.
We also denote the base change of $\pi_0 \colon J^{0}(S) \to C$ by the morphism $\Spec \mc{O}_{C,x}^h\to C$ by 
$$
J^0_x\to \Spec \mc{O}_{C,x}^h.
$$
Then it is known by \cite[Proposition 5.4.3 in p.314, Theorem 5.4.3 in p.321]{MR986969}
 that there exists an exact sequence:
\begin{equation}\label{eqn:brauer}
\begin{array}{ccccccc}
0&\to&  \Br(J^0(S))&\to& WC(J^0_\eta)&\to&  \bigoplus _{x\in C}WC(J^0_x)\\
 &     &&&\rotatebox{90}{$\in$}&&\rotatebox{90}{$\in$}\\
&&&& \xi                       &\mapsto& (\xi_x)_{x\in C}
\end{array}
\end{equation}
Here, we denote  the image  of $\xi$ (given in \eqref{eqn:xi}) in $WC(J^0_x)$ by $\xi_x$. It follows from  \cite[Proposition 5.4.2]{MR986969} that  $m_x=\ord \xi_x$,
where $m_x$ is the multiplicity of the fiber of $\pi$ over the point $x\in C$.
Define 
\begin{equation}\label{eqn:lambda'}
\lambda'_\pi:=\lcm _{x\in C}(m_x)=\ord ((\xi_x)_{x\in C}).
\end{equation}
Since $\ord \xi$ is divided by $\ord ((\xi_x)_{x\in C})$, we see from \eqref{eqn:order_xi} that 
$$
\lambda'_\pi \mid \lambda_\pi.
$$

In particular, if $i\in \Z$ is coprime to $\lambda_{\pi}$, then $i$ is coprime to each $m_x$, and thus we have 
\begin{equation}\label{eqn:ord_iximx}
\ord(i\xi)_x=\ord i(\xi_x)=\ord(\xi_x)=m_x.
\end{equation}
Combining \eqref{eqn:ord_iximx} with \eqref{eqn:ixi}, we know that
the multiplicity of the fiber of $\pi_i$ over the point $x$ is also $m_x$. This shows the first statement of Lemma \ref{lem:SJS} (i). 

Define a subgroup $H'_{\pi}$ of  the group $H_\pi(:=\{ i\in (\Z/\lambda_{\pi}\Z)^* \mid J^i(S)\cong S\}$ given in \eqref{eqn:Hpi}) to be
\begin{align}\label{ali:H'}
H'_{\pi}:=&\{  i \in H_{\pi}   \mid i\equiv 1 \ (\mod \ \lambda'_\pi)  \}.
\end{align}
We use the following lemma to obtain a lower bound of the cardinality  of the set $\FM(S)$.

\begin{lemma}\label{lem:H_1H_2}
Let $\pi\colon S\to C$ be an elliptic surface with $\Br (J^0(S))=0$.
Then we have 
$$ \bigl| H_{\pi}/H'_{\pi}\bigr| \le   \bigr| \Aut_0(J^0_{\eta})\bigr|.$$
\end{lemma}

\begin{proof}
For each $i\in H_{\pi}$,   
fix an isomorphism $\theta_i \colon J^1_{\eta} \to J^i_{\eta}$ over the generic point $\eta \in C$.
Then we obtain a structure of a homogeneous space  on $J^1_{\eta}$ by the action 
\[
\mu_i' := \theta_i^{-1} \circ \mu_i \circ (\theta_i \times \rid_{J^0_{\eta}}) \colon J^1_{\eta} \times J^0_{\eta} \to J^1_{\eta}
\]
such that $(J^i_{\eta}, \mu_i) = (J^1_{\eta}, \mu_i')$ holds in $WC(J^0_{\eta})$ by the definition. 
On the other hand, by \cite[Exercise 10.4]{Si09}, $(J^1_{\eta}, \mu_i') = (J^1_{\eta}, \mu_1 \circ (\id_{J^1_{\eta}}\times \phi))$ for some $\phi \in \Aut_0(J^0_{\eta})$. 
We define an equivalence relation $\sim$ of $\Aut _0(J^0_{\eta})$
such that $$\phi_1\sim \phi_2$$ 
for $\phi_i\in \Aut _0(J^0_{\eta})$
when
$$ 
(J^1_{\eta}, \mu_1 \circ (\id_{J^1_{\eta}}\times \phi_1))
 = (J^1_{\eta}, \mu_1 \circ (\id_{J^1_{\eta}}\times \phi_2)).
$$ 
Then we can define a map 
$$
f \colon H_{\pi}\to  \Aut_0(J^0_{\eta})/{\sim} \quad  i\mapsto \phi.
$$
We  see that $ij^{-1}\in H'_{\pi}$ if and only if $f(i)=f(j)$ as follows. First note that we have an injection 
$$
WC(J^0_\eta)\hookrightarrow\bigoplus _{x\in C}WC(J^0_x) \qquad
\xi=(J^1_{\eta},\mu_1) \mapsto  (\xi_x)_{x\in C}
$$
by the vanishing of the Brauer group $\Br(J^0(S))$ and \eqref{eqn:brauer},
and hence 
\begin{equation}\label{eqn:xi_lambda}
\ord \xi=\lambda'_\pi (:=\ord ((\xi_x)_{x\in C})).
\end{equation}
We observe that  for $i,j\in H_{\pi}$, the condition $f(i)=f(j)$ is equivalent to the equality 
$i\xi = j\xi$ by \eqref{eqn:ixi}, which is also equivalent to $i^{-1}j\in H'_{\pi}$ by \eqref{eqn:xi_lambda}.

Consequently, we obtain an inclusion
$$
H_{\pi}/H'_{\pi}\hookrightarrow \Aut_0(J^0_{\eta})/{\sim}
$$
and the conclusion.
\end{proof}

%%%%%%%%%%%%%%%%%%%%%%%%%%%%%%%%%%%%%%%%%%%%%%%%%
\section{Elliptic curves and automorphisms}\label{sec:elliptic} 

Let $F$ be an elliptic curve over an algebraically closed field $k$ with $p= \ch k \ge 0$. 
The explicit description of the automorphism group  $\Aut _0(F)$ fixing the origin $O$ is well-known, and is given as follows. 
\begin{theorem}[cf.~Appendix A in \cite{Si09}]\label{thm:AutF}
The automorphism group $\Aut _0(F)$ is
\begin{align*}
&\Z/2\Z                          &\mbox{ if } &j(F)\ne 0,1728,     \\
&\Z/4\Z                          &\mbox{ if } &j(F)=1728 \mbox{ and } p\ne 2,3,\\
&\Z/6\Z                          &\mbox{ if } &j(F)=0 \mbox{ and } p\ne 2,3,  \\
&\Z/3\Z \rtimes \Z/4\Z   &\mbox{ if } &j(F)=0=1728 \mbox{ and } p=3,\\
&Q \rtimes \Z/3\Z        &\mbox{ if } &j(F)=0=1728 \mbox{ and } p=2.
\end{align*}
Note that
in the last second case, $\Z/4\Z$ acts on $\Z/3\Z$ in the unique non-trivial way, and
in the last case, the group is so called a binary tetrahedral group, and $Q$ is the quaternion group.
In the last two cases $F$ is necessarily supersingular. 
\end{theorem}

For points $x_1, x_2 \in F$, to distinguish the summation of divisors and of elements in the group scheme $F$, we denote by $x_1 \oplus x_2$ the sum of them by the operation of $F$, and
\[
i \cdot x_1 := x_1 \oplus \cdots \oplus x_1 \quad \mbox{($i$ times)}. 
\]
Furthermore, we use the symbol $T_a$ to stand for the translation by $a\in F$:
$$
T_a\colon F\to F \quad x\to a\oplus x.
$$
We also denote by
\[
ix_1 := x_1 + \cdots + x_1 \quad \mbox{($i$ times)}
\]
the divisors on $F$ of degree $i$. 
We denote the dual abelian variety $\Pic^0F$ of $F$ by $\hat{F}$. 
It is well-known that there exists a group scheme isomorphism
\begin{equation}\label{eqn:dual}
F \to \hat{F} \quad x \mapsto \cO_F(x-O),
\end{equation}
where $O$ is the origin of $F$. 

We will use the following lemma several times. 

\begin{lemma}\label{lem:action}
Take a point $a \in F$ with $\ord(a) \ge 4$, and $\phi \in \mathrm{Aut}_0(F)$. If $\phi(a) = a$, then $\phi = \rid_F$. 
\end{lemma}

\begin{proof}
%At first, note that $g^i(a) = a$ for any $i \in \Z$, since $g(a) = a$. 
In any of the cases in Theorem \ref{thm:AutF}, we have $\ord(\phi)\in \{1, 2, 3, 4,6\}$. Let us first consider the case $\ord(\phi)= 2, 4$ or $6$. 
In this case, $\phi^i = -\rid_F$ for some $i \in \Z$, and hence we get $-1\cdot a = a$.
The condition $\ord(a) \ge 4$  yields a contradiction. 
Next, consider the case $\ord(\phi) = 3$. Then we have
\[
(\phi-\rid_F)(\phi^2+\phi+\rid_F)=0
\]
in the domain $\End(F)$, which implies that $\phi^2+\phi+\rid_F=0$, and hence 
$\phi^2(a) \oplus \phi(a) \oplus a = O$. By the assumption $\phi(a)=a$, we see that $3 \cdot a = O$. This is absurd by $\ord(a) \ge 4$. 
\end{proof}

For a non-zero integer $m$, we define the $m$-torsion subgroup of $F$ to be 
\[
F[m] := \{ a \in F \mid m \cdot a = O \}. 
\]
Equivalently, $F[m]$ is the kernel of the multiplication map by $m$. 
Recall that  
\begin{equation*}
F[m]= \begin{cases}
		\Z/p^e\Z & \text{if $F$ is ordinary, $m=p^e$, $e>0$} \\
		\{O\} & \text{if $F$ is supersingular, $m=p^e$, $e>0$}\\
		\Z/m\Z\times \Z/m\Z &  \text{if $p \nmid m$}.
	\end{cases}
\end{equation*} 
(See \cite[Corollary III.6.4]{Si09}.) Note that these $3$ cases do not exhaust all possibilities (i.e., cases where $m$ is divisible by $p$ but is not power of $p$ is not covered.)

Take $a \in F$ with $\ord(a) = m$. %Suppose that $p=0$ or $p>0$ does not devide $m$. 
In order to count Fourier--Mukai partners of elliptic ruled surfaces, we need to study the subgroup 
\begin{equation}\label{eqn:def_H}
H^a_F := \{ i \in (\Z/m\Z)^* \, | \, \exists \phi \in \Aut_0(F) \, \text{such that} \, \phi(a) = i \cdot a \} 
\end{equation}
of  $(\Z/m\Z)^*$. Note that the definition of $H^\mc{L}_{\hat{E}}$ given in \eqref{eqn:def_H00}
is compatible with \eqref{eqn:def_H}.
We obtain the following result as a direct consequence of  Lemma \ref{lem:action}.
\begin{lemma}\label{lem:injection}
Take $a \in F$ with $\ord(a)  \ge 4$.  
\begin{enumerate}
\item
We have an injective group homomorphism
\begin{equation}\label{eqn:HaF}
\iota\colon H^a_F \hookrightarrow \Aut_0(F). 
\end{equation}
Furthermore, we have
$|H^a_F|=2,4$ or $6$.
\item
Suppose that $p>0$ and $\ord(a)=p^e$.
Then \eqref{eqn:HaF} is an isomorphism. 
\end{enumerate}
\end{lemma}
\begin{proof}
(i) Take $i\in H^a_F$. Then there exists $\phi\in\Aut_0(F)$ such that $\phi(a)=i\cdot a$, and define $\iota (i)$ to be $\phi$. The well-definedness of $\iota$ follows from  Lemma \ref{lem:action}, and $\iota$ is injective by the definition.
Since $H^a_F$ is regarded as an abelian subgroup of $\Aut_0(F)$ described in Theorem \ref{thm:AutF}, and $H^a_F$ contains $\{\pm 1\}$ as a subgroup, we obtain the second assertion. 

(ii) The existence of an order $p^e$ element in $F$ implies that $F$ is ordinary.
Since $F[p^e] = \Z/p^e\Z = \langle a \rangle$, for any $\phi \in \Aut_0(F)$ we see that $\phi(a) =i\cdot a$ for some $i \in (\Z/p^e\Z)^*$. Hence the injective  homomorphism in \eqref{eqn:HaF} is surjective, and then 
we can confirm the statement.
\end{proof}
From now on, by \eqref{eqn:HaF} we often regard $H^a_F$ as a subgroup of $\Aut_0(F)$ when $\ord a\ge 4$. 

%%%%%%%%%%%%%%%%%%%%%%%%%%%%%%%%%%%%%%%%%%%%%%%%%%
\section{Pirozhkov's result and its application}\label{sec:Pirozhkov}
%%%%%%%%%%%%%%%%%%%%%%%%%%%%%%%%%%%%%%%%%%%%%%%%%%

In this section, we summarize some definitions and results in \cite{Pi20b}, and give their application to the Popa--Schnell conjecture. 
We also refer to \cite{Pe19} for fundamental notions of $\infty$-categories. 

For a Noetherian scheme $S$ over $k$, we denote by $\mathrm{Perf}(S)$ the full subcategory of $D^b(S)$ consisting of perfect complexes.
A stable $k$-linear $\infty$-category $\cD$ is said to be $S$-linear if there exists an action functor 
\[
a_{\cD} \colon \cD \times \mathrm{Perf}(S) \to \cD
\]
together with associativity data. 
%For a precise definition see \cite{Pe19}, but in this paper, we do not use the precise definition explicitly. 

For a morphism $f \colon X \to S$ between smooth projective varieties $X$ and $S$ over $k$, the category $D^b(X)$ has a natural $S$-linear structure via the functor 
$$
D^b(X) \times D^b(S) \to D^b(X) \quad (\mc{E},\mc{F}) \mapsto \mc{E} \Lotimes_{X} \mathbb{L}f^*\mc{F}.
$$

\begin{definition}[\cite{Pi20b}]
\label{definition:NSSI}
Let $S$ be a Noetherian scheme over a field $k$. We say that $S$ is \emph{noncommutatively stably semiorthogonally indecomposable}, or \emph{NSSI} for brevity, if for arbitrary choices of 
\begin{enumerate}
\item $\mathcal{D}$, a $S$-linear category which is proper\footnote{See \cite{Pe19} for this notion.}  over $S$ and has a classical generator, and
\item $\mathcal{A}$, a left admissible subcategory  of $\mathcal{D}$ which is linear over $k$, 
\end{enumerate}
the subcategory $\mathcal{A}$ is closed under the action of $\mathrm{Perf}(S)$ on $\mathcal{D}$. 
\end{definition}

\begin{remark}
For a quasi-compact and quasi-separated scheme $S$, the category $\mathrm{Perf}(S)$ has a classical generator by \cite[Corollary 3.1.2]{BB03}. In particular, for a smooth projective variety $S$, the category $D^b(S)$ has a classical generator. 
\end{remark}

\begin{theorem}[Lemma 6.1 in \cite{Pi20b}]
\label{theorem:fibrations}
Let $\pi \colon X \to S$ be a smooth projective morphism which is an \'{e}tale-locally trivial fibration with fiber $X_0$. Assume that $S$ is a connected excellent scheme\footnote{In \cite[Lemma 6.1]{Pi20b}, Pirozhkov assumes that $S$ is a scheme over $\mathbb{Q}$, but it is not needed in its proof.}.% and that $\pi$ admits a relative ample invertible sheaf. 
Then for any point $s \in S$ the base change map
\[
\begin{Bmatrix}
\textup{$S$-linear admissible} \\ \textup{subcategories} \\ \mathcal{A} \subset D^b(X)
\end{Bmatrix}
\xrightarrow{\textup{restriction to $X_s \cong X_0$}} 
\begin{Bmatrix}
\textup{admissible subcategories} \\ \mathcal{A}_0 \subset D^b(X_0)
\end{Bmatrix} 
\]
is an injection. 
\end{theorem}

\begin{definition}
Let $\pi \colon X \to S$ be a smooth projective morphism of Noetherian schemes. 
\begin{enumerate}        
\item
An object $\mc{E}\in \mathrm{Perf} (X)$ is $\pi$-\emph{exceptional} if 
$\mathbb{R}\pi_*\mcRHom _X(\mc{E},\mc{E})\cong\mc{O}_S$.
\item
A collection of $\pi$-exceptional objects
\(
    \mc{E}_{ 1 }, \dots, \mc{E} _{ N } \in \mathrm{Perf} (X)
\)
is a $\pi$-\emph{exceptional collection} if
$\mathbb{R} \pi_{ \ast } \RHom ( \cE _{ j }, \cE _{ i } ) = 0$
for any $1 \le i < j \le N$.
\item
A $\pi$-\emph{exceptional pair} is a $\pi$-exceptional collection of length 2.
\end{enumerate}
\end{definition}

For a $\pi$-exceptional pair
\(
    \mc{E}, \mc{F}
\),
the left $\pi$-mutation
\(
    L_{ \mc{E} } \mc{F}
\)
of \( \mc{F} \) through \( \mc{E} \) 
and the right $\pi$-mutation
\(
    R _{ \mc{F} } \mc{E}
\)
of \( \mc{E} \) through \( \mc{F} \) are defined by the following distinguished triangles:
\begin{align*}
     \pi ^{ \ast } \mathbb{R} \pi _{ \ast } \mcRHom _{ X } ( \mc{E}, \mc{F} ) 
     \otimes _{ \mc{O} _{X } } \mc{E}
     \xrightarrow{ \varepsilon }
     \mc{F}
     \to
     L _{ \mc{E} } \mc{F}\\
     R _{ \mc{F} } \mc{E}
     \to
     \mc{E}
     \xrightarrow{ \eta }
     \pi ^{ \ast } \mathbb{R} \pi _{ \ast } \mcRHom _{X} ( \mc{E}, \mc{F} ) ^{ \vee } \otimes _{ \mc{O} _{X} } \mc{F}
\end{align*}
We  see that mutations commute with base change.

\begin{lemma}[Lemma 2.22 in \cite{IOU21}]\label{lm:base change}
Consider the following Cartesian square of finite dimensional Noetherian schemes, where $\pi$ is smooth projective.
\[ 
\xymatrix{ 
Y\ar[r]^f \ar[d]_{\varphi} &   X    \ar[d]^{\pi} \\
T\ar[r]_g &  S  
}
\]
   For any \( \pi \)-exceptional pair
    \(
        \left( \cE, \cF \right)
    \), it follows that
    \(
        \left( f ^{ \ast } \cE, f^{ \ast }\cF \right)
    \)
    is an \( \varphi \)-exceptional pair and we have the following isomorphisms:
\begin{align*}
    L _{f ^{\ast}\cE} \left(f ^{\ast}\cF\right) \simeq f ^{\ast}( L _{\cE } \cF )\\
    R _{f ^{\ast}\cF} \left(f ^{\ast}\cE \right) \simeq  f^{\ast}( R _{\cF } \cE )
\end{align*}
\end{lemma}

We apply Theorem \ref{theorem:fibrations} and Lemma \ref{lm:base change} to obtain the following.

%%%
%%%
%%%

\begin{proposition}\label{prop:adm_on_proj}
Let $\pi\colon X \to S$ be a $\mathbb{P}^n$-bundle ($n=1,2$) over a smooth projective variety $S$.
Then any non-trivial $S$-linear admissible subcategory of $D^b(X)$ is of the following form:
\begin{enumerate}
\item
(Case $n=1$)
$$D^b(S)(i) (:= \pi^*D^b(S)\otimes_{\mc{O}_X}\mc{O}_X(i))$$ for some $i\in \Z$.
\item
(Case $n=2$)
$$\pi^*D^b(S)\otimes_{\mc{O}_X}\Span{\mc{E}_1,\ldots,\mc{E}_l},$$
where $\mc{E}_1,\ldots,\mc{E}_l$ ($1\le l\le n+1$) is a $\pi$-exceptional collection. 
\end{enumerate}
\end{proposition}

\begin{proof}
(i) Any non-trivial admissible subcategory in $D^b(\mathbb{P}^1)$ is known to be of the form 
$\Span{\mc{O}_{\mathbb{P}^1}(i)}$ for some $i\in \Z$.
Since the restriction of the admissible category 
$D^b(S)(i)$ to a fiber is $\Span{\mc{O}_{\mathbb{P}^1}(i)}$, 
the injective base change map in  Theorem \ref{theorem:fibrations} is surjective.
Hence the result follows.

(ii)
\cite[Theorem 4.2]{Pi20a} states that any non-trivial admissible subcategory $\mathcal{A}$ in $D^b(\mathbb{P}^2)$ is generated by a subcollection of successive 
mutations of the standard exceptional collection 
$\mc{O}_{\mathbb{P}^2}, \mc{O}_{\mathbb{P}^2}(1), \mc{O}_{\mathbb{P}^2}(2).$ 
Lemma \ref{lm:base change} yields an $S$-linear admissible subcategory $\mathcal{A}_X$ of $D^b(X)$, which  
is generated by a $\pi$-exceptional subcollection obtained by successive 
$\pi$-mutations of the $\pi$-exceptional collection 
$\mc{O}_X, \mc{O}_X(1), \mc{O}_X(2)$, and its derived restriction on a fiber is 
$\mathcal{A}$.  This means that
 the injective base change map in Theorem \ref{theorem:fibrations} is surjective,
hence we  obtain the result.  
\end{proof}

The Popa--Schnell conjecture in \cite{PS11} states that for any Fourier--Mukai partners $X'$ of a given 
smooth projective variety $X$, there exists an equivalence 
$D^b(\Alb (X'))\cong D^b(\Alb (X))$ of derived categories. 

From Proposition \ref{prop:adm_on_proj},  we deduce that the Popa--Schnell conjecture
holds true in certain situations. 

%%%
%%%
%%%

\begin{corollary}
\label{corollary:base-elliptic}
Let $X \to A$ and $X' \to A'$ be $\mathbb{P}^n$-bundles over abelian varieties $A$ and $A'$ for $n=1,2$. If $X$ and $X'$ are Fourier--Mukai partners, then so are $A$ and $A'$. Furthermore, the Popa--Schnell conjecture holds true in this case.
\end{corollary}

\begin{proof}
Put $ D^b(A)(i)=\pi^* D^b(A)\otimes \cO_{X}(i)$, where $\pi$ is the $\mathbb{P}^1$-bundle $X \to A$.
Since abelian varieties are NSSI by \cite[Theorem 1.4]{Pi20b}, 
any admissible category of $D^b(X)$ is $A$-linear.
Proposition \ref{prop:adm_on_proj} implies that any non-zero indecomposable admissible subcategory of $D^b(X)$ is equivalent to $D^b(A)$.  
This completes the proof of the first assertion. We  see that $A\cong \Alb(X)$ and $A'\cong \Alb (X')$, and hence obtain the second. 
\end{proof}

If $X$ is an elliptic ruled surface over $\bbC$, namely  $n=1$ and $k=\bbC$, in Corollary \ref{corollary:base-elliptic}, the statement follows from \cite[Theorem 1.1]{Ue17}. The proof given above for $n=1,2$ and arbitrary $k$  
is more direct and natural.  

\begin{remark}
Let  $X \to E$ and $X' \to E'$ be $\mathbb{P}^n$-bundles over elliptic curves $E$ and $E'$ for $n=1, 2$. 
As a consequence of  Corollary \ref{corollary:base-elliptic},
if $X$ and $X'$  are Fourier--Mukai partners, then  
$D^b(E)\cong D^b(E')$, and hence $E\cong E'$ by \cite[Corollary 5.46]{Hu06}.
\end{remark}
%%%%%%%%%%%%%%%%%%%%%%%%%%%%%%%%%%%%%%%%%%%%%%%%%%%%%%%%%%%%

\section{Fourier--Mukai partners of elliptic ruled surfaces}\label{sec:FM}

%%%%%%%%%%%%%%%%%%%%%%%%%%%%%%%%%%%%%%%%%%%%
\subsection{Singular fibers of elliptic ruled surfaces}\label{subsec:FM_ruled}
In this subsection, we recall a result in \cite{TU22}. 
Let $\mc E$ be a normalized, in the sense of \cite[Ch.~5.~\S 2]{Ha77}, rank $2$ vector bundle 
on an elliptic curve $E$ and 
$$
f \colon S= \mathbb{P} (\mathcal{E} ) \rightarrow E
$$
be a $\PP ^1$-bundle on $E$ defined by $\mc E$. 
Let us put $e := -\deg \mathcal{E}$. 
If $S$ has an elliptic fibration, then $-K_S$ is nef. Then we can easily deduce  $e=0$ or $-1$ from \cite[Corollary V.2.11, Theorems V.2.12, V.2.15]{Ha77}).

\begin{theorem}[Theorem 1.1 in \cite{TU22}]\label{TUmain}
Let us consider the above situation. 
\begin{enumerate}
\item
For $e = 0$, we have the following possibilities:
\begin{center}
\begin{tabular}{|c|c|c|c|}
\hline
       & $\mathcal{E}$ & $\exists$ an elliptic fibration on $S$?  & $p$ \\ \hline \hline
(i-1)& $\mc{O}_E \oplus \mc{O}_E$ &no multiple fibers & $p\ge 0$ \\ \hline
(i-2)& $\mc{O}_E \oplus \mathcal{L}$,  $\ord \mc L=m>1$ &$(m,m)$  & $p\ge 0$ \\ \hline
(i-3)& $\mc{O}_E \oplus \mathcal{L}$,  $\ord \mc L=\infty$ & no elliptic fibrations &$p\ge 0$\\ \hline
(i-4)& indecomposable & no elliptic fibrations & $p=0$ \\ \hline
(i-5)& indecomposable & $({p}^*)$ & $p>0$ \\ \hline
\end{tabular} 
\end{center}

Here $\mc L$ is an element of $\Pic ^0E$.
In the case $S$ has an elliptic fibration $\pi$,  for example, the notation $(m,m)$ in (i-2) means that $\pi$ has exactly two multiple fibers of multiplicities $m$. 
\item
Suppose that $e=-1$. Then the isomorphism class of such vector bundle $\mc{E}$ on $E$ is unique, and $S$ has an elliptic fibration. The list of singular fibers are as follows: 
\begin{center}
\begin{tabular}{|c|c|c|c|}
\hline
       & multiple fibers  &E                                   & $p$ \\ \hline \hline
(ii-1)&$(2,2,2)$     &                                      & $p\ne 2$ \\ \hline
(ii-2)&$(2^*)$   &  \mbox{supersingular}      & $p= 2$\\ \hline
(ii-3)&$(2, 2^*)$ & \mbox{ordinary}         & $p= 2$\\ \hline
\end{tabular} 
\end{center}
\end{enumerate}
The symbol ${}^*$ stands for a wild fiber in the tables. 
\end{theorem}

By \cite{BM01} and \cite{Ka02}, we know that if $S$ has non-trivial Fourier--Mukai partners, $S$ has an elliptic fibration. Hence, from now on, we suppose that $S$ has an elliptic fibration $\pi\colon S\to \PP^1$. Theorem \ref{TUmain} says that the multiplicities of all multiple fibers of $\pi$ are the same number  $m$. 

When $e=0$ (resp. $e=-1$), 
we see
\begin{equation}\label{eqn:FF}
F_\pi\cdot F_f=mC_0\cdot F_f=m \quad \mbox{
(resp. } F_\pi\cdot C_0=m(2C_0-F_f)\cdot C_0=m)
\end{equation}
by \cite[Remark 4.2]{TU22}, 
and hence 
\begin{equation}\label{eqn:pipi'}
\lambda_\pi=m=\lambda_{\pi}'
\end{equation}
for both cases (recall the definitions of $\lambda_\pi$ and $\lambda_\pi'$ in \eqref{eqn:pi} and \eqref{eqn:lambda'} respectively). Here $F_\pi$ (resp.~$F_f$) is a fiber of $\pi$ (resp.~$f$), and $C_0$ stands for a section of $f$ satisfying $C_0^2=-e$. 

Consider the case $|\FM(S)|\ne 1$. Then the inequality \eqref{eqn:FMeqn} yields $m=\lambda_{\pi}\ge 5$.
Hence, $S$ fits into either (i-2), $m\ge 5$ or (i-5), $p\ge 5$ in Theorem \ref{TUmain}.
Then $S'\in \FM(S)$ is also an elliptic ruled surface admitting an elliptic fibration $\pi'$ fitting into the same case as $S$ by Lemma \ref{lem:SJS}. 

\begin{lemma}\label{lem:i-5}
Suppose that $|\FM(S)|\ne  1$. Then $S$ fits into the case (i-2).
\end{lemma}

\begin{proof}
It suffices to show that $|\FM(S)|=1$ in the case (i-5). 
Suppose that $S$ fits into the case (i-5).
As we explained above, $S'\in \FM(S)$ is also an elliptic ruled surface in  the case (i-5). In other words,  $S'$ has a $\PP^1$-bundle structure $f'\colon \PP(\mc{E}')\to E'$, where
$\mc{E}'$ is the indecomposable vector bundle of rank $2$, degree $0$ on an elliptic curve $E'$. 
By Corollary \ref{corollary:base-elliptic}, we have $E\cong E'$. 
Then, we see $S\cong S'$ by \cite[Theorem V.2.15]{Ha77}, in other words, $|\FM(S)|=1$.
\end{proof}

The purpose of this paper is to describe the set $\FM(S)$ for elliptic ruled 
surfaces. Hence in the sequel, we will concentrate on the case (i-2),
the unique candidate of $S$ admitting non-trivial Fourier--Mukai partners.

\subsection{Case (i-2).}\label{subsec:i-2}
Take $\mc{L} \in \Pic^0 E$ with $1<m:=\ord \mc{L}<\infty$, and set 
$$S:=\PP(\cO_E \oplus \cL).$$

The following lemma is elementary and useful.
%%%
%%%
%%%
\begin{lemma}\label{lem:automor}
\begin{enumerate}
\item
There exists an isomorphism $S \cong \PP(\cO_E \oplus \cM)$ over $E$ if and only if $\cL\cong \cM^{\pm 1}$.  
\item
For $\phi_E\in \Aut (E)$, 
we have an isomorphism $f^*\phi_E$ in the fiber product diagram:
\begin{equation}\label{eqn:SE_fiber_product}
\xymatrix{
\PP(\cO_E \oplus \phi_E^*\mc L) \ar[d]\ar[r]^{\qquad f^*\phi_E} & S \ar[d]^f  \\
E \ar[r]_{\phi_E}& E\ar@{}[lu]|{\Box}}
\end{equation}
\item
For some $\cM\in\Pic^0 E$, let $f_T\colon T:=\PP(\cO_E \oplus \cM)\to E$ be the $\PP^1$-bundle over $E$.
Suppose that we are given an isomorphism $\phi\colon T \to S$.
Then,  
if we replace $\phi$ appropriately, we can take
$\phi_E\in \Aut_0(E)$,
which makes the diagram
\begin{equation}\label{eqn:SEE}
\xymatrix{
T \ar[d]_{f_T}\ar[r]^{\phi} & S \ar[d]^f \\
E \ar[r]_{\phi_E}& E
}
\end{equation}
commutative. Moreover we have an isomorphism 
\begin{equation}\label{eqn:TPP}
T\cong \PP(\cO_E \oplus \phi_E^*\cL)
\end{equation}
over $E$, and an isomorphism
\begin{equation}\label{eqn:MEL}
\cM\cong \phi_E^*\cL. 
\end{equation}
\end{enumerate}
\end{lemma}

%%%%%%%%%%%%%%%%

\begin{proof}
(i) This fact directly follows from \cite[Exercise II.7.9(b)]{Ha77}.
\newline
\ \ \ (ii)
This assertion must be well-known. We leave the proof to readers.
(For example, use \cite[Proposition II.7.12]{Ha77}.) 
\newline
\ \ \ (iii) Since 
$S$ has a unique $\PP^1$-bundle structure, the existence of $\phi_E\in \Aut (E)$ fitting in \eqref{eqn:SEE} is assured. 
Next, write $\phi_E= T_a\circ \phi^0_E$ for some 
$\phi^0_E\in \Aut_0(E)$ and $a\in E$.  
Since $T_a^*\cL\cong \cL$, the isomorphism $f^*T_a$ (given as $f^*\phi_E$ in \eqref{eqn:SE_fiber_product}) gives an automorphism of $S$. 
Then, if necessary, replace $\phi$ with $(f^*T_a)^{-1}\circ \phi$, we may assume that $\phi_E\in \Aut_0(E)$.
By the universal property of the fiber product in \eqref{eqn:SE_fiber_product}, we obtain an isomorphism 
\eqref{eqn:TPP} over $E$.
 Then by (i) there exists an isomorphism $\cM^{\pm 1}\cong \phi_E^*\cL$. Since $(-\id_E)^*\cL\cong \cL^{-1}$,
 $f^*(-\id_E)$ also gives an automorphism of $S$. Thus, replace $\phi$ with 
  $f^*(-\id_E)\circ \phi$ if necessary,  we may assume that $\phi_E\in \Aut_0(E)$ and \eqref{eqn:MEL} holds simultaneously.
\end{proof}

%%%
%%%
%%%
\begin{lemma}\label{lem:projbdl}
For $i \in (\Z/m\Z)^*$, $S \cong \PP(\cO_E \oplus \cL^i)$ if and only if there exists an automorphism $\phi_E \in \Aut_0(E)$ such that $\phi_E^* \cL \cong \cL^{i}$. 
Consequently, the set
$$
\{ \PP ( \mc O_{E}\oplus \mc L^i)   \mid i \in (\Z/m\Z)^* \}/\cong
$$
is naturally identified with the group 
$$
(\Z/m\Z)^*/ H^{\cL}_{\hat{E}}.
$$
Here, recall that 
$H^{\cL}_{\hat{E}} := \{ i \in (\Z/m\Z)^* \, | \, \exists \phi \in \Aut_0(E) \, \text{such that} \, \phi^*\cL \cong \cL^{i}\}$.
\end{lemma}
\begin{proof}
``If'' part follows from Lemma \ref{lem:automor} (ii). 
 ``Only if'' part follows from  Lemma \ref{lem:automor}  (iii).
\end{proof}

Consider the dual morphism 
\begin{equation}\label{eqn:q1}
q_{1}\colon F_0:=\widehat{\hat{E}/\Span{\mc{L}}}\to E
\end{equation}
of the quotient morphism $\hat{E}\to \hat{E}/\Span{\mc{L}}$. Then it follows from the definition of $q_1$ that 
$q_1^*\cL\cong \mc{O}_{F_0}$ holds. 
Thus we have a diagram 
\begin{equation}\label{eqn:pull_back}
\xymatrix{
F_0 \ar[d]_{q_{1}} & F_0\times \PP^1 \ar[l]_{p_1} \ar[r]^{p_2} \ar[d]^{q_S} & \PP^1 \ar[d]^{q_{2}} \\
E & S \ar[r]_{\pi} \ar[l]^{f} \ar@{}[lu]|{\Box} & \PP^1,\\
}
\end{equation}
where the left square diagram is a fiber product, and the right one is obtained by the Stein factorization of $\pi\circ q_S$. The reason why $\pi\circ q_S$ factors through $p_2$ is as follows.
First, we have $q_S^*\omega_S\cong \omega_{F_0\times \PP^1}$ by \cite[Lemma 2.14]{TU22}. On the other hand, the elliptic fibration $p_2$ (resp.~$\pi$) are defined by the linear system of some multiple of $-K_{F_0\times \PP^1}$ (resp.~$-K_S$). Therefore $\pi\circ q_S$ factors through $p_2$. 

Recall that the elliptic fibration $\pi$ has exactly two multiple fibers.
\vskip\baselineskip

\noindent
\textbf{Convention.}
By the action of $\PGL (1,k)$ on $\PP^1$, we always assume below that in the case (i-2), the elliptic fibration $\pi$ has multiple fibers over the points 
 $0$ and $\infty$ in $\PP^1$. Furthermore, we also assume that $q_2(0)=0$ and $q_2(\infty)=\infty$. 
 
\vskip\baselineskip
\noindent
For $y_0\in \PP^1$ with $y:=q_{2}(y_0)\in \PP^1\backslash \{0,\infty\}$, we denote by $F_y$  the non-multiple fiber  of  $\pi$ over the point $y$.  Then it follows from $f\circ q_S=q_1\circ p_1$ that the restriction of $q_S$ induces the isomorphism 
\begin{equation}\label{eqn:hi}
q_S|_{F_0\times y_0}\colon F_0\times y_0 \cong F_y,
\end{equation}
since we see from \eqref{eqn:FF} that $f|_{F_y}$ is finite morphism of  degree $m$.
We tacitly identify $F_0$ and $F_y$ by this isomorphism.  

Take $x_0\in F_0$ and set $x:=q_1(x_0)\in E$. Then 
in a similar way to \eqref{eqn:hi}, we have an isomorphism
\begin{equation}\label{eqn:PP}
q_S|_{x_0\times \PP^1}\colon x_0\times \PP^1\cong F_x,
\end{equation}
where $F_x$ is the fiber of $f$ over the point $x$.
We identify $\PP^1$ and $F_x$ by \eqref{eqn:PP}.  
By our convention above, we see that the two multiple fibers of $\pi$ intersect with each fiber $\PP^1$  of $f$ at $0$ and $\infty$
respectively. 

Recall that $f$ has two minimal sections, let's say $C_0$ and $C_1$, corresponding to the projections 
\begin{equation}\label{eqn:OLO}
\mc{O}_E\oplus \mc{L}\to \mc{O}_E \quad\mbox{ and }\quad  \mc{O}_E\oplus \mc{L}\to \mc{L}.
\end{equation}
Then the multiple fibers of $\pi$ are given exactly $mC_0$ and $mC_1$ (see \cite[Remark 4.2]{TU22}). 

We use the following lemma to show Claim \ref{cla:action}.
%%%%
\begin{lemma}\label{lem:autoPP}
Let us regard the multiplicative group $\G_m$ as a subgroup of $\Aut(\cO_E \oplus \cL)(\cong \G_m\times \G_m)$ by the diagonal embedding. Then 
there exists an injective homomorphism 
\begin{equation*}\label{eqn:injSE}
\iota\colon \G_m\cong \Aut(\cO_E \oplus \cL)/\G_m\hookrightarrow \Aut (S/E).
\end{equation*}
Here, for $\lambda\in \G_m$, the automorphism $\iota(\lambda)$ of $S$ induces the action on each fiber $\PP^1$ of $f$ fixing the points $0$ and $\infty$.
% If we can regard that $\G_m$ acts on $\G_m(\cong \PP^1\backslash \{0,\infty\}),$  $\lambda$ acts on $\G_m$ by the scalar multiplication.   
\end{lemma}
\begin{proof}
The existence of the injection $\iota$ is assured in \cite[p.202]{MR1603467}.\footnote{See also \cite[Lemma 3]{MR280493}). Because $\Delta$ in ibid. is trivial, we actually see that $\iota$ gives an isomorphism.}
Note that since any elements of $\Aut(\cO_E \oplus \cL)$ preserve the projections in \eqref{eqn:OLO}, any $\beta\in \Im\iota$ 
preserves the minimal sections $C_0$ and $C_1$, and hence it gives an automorphism on each fiber $\PP^1$ of $f$  fixing the points $0$ and $\infty$.
\end{proof}

\subsection{Proof of Theorem \ref{thm:main}.}
Let $S$ be an elliptic ruled surface and suppose $|\FM(S)| \neq 1$. Lemma \ref{lem:i-5} implies that 
$$S \cong \PP(\mc{O}_{E} \oplus \mc{L})$$
for some $\mc{L} \in \Pic^0 E$ with $\ord \mc{L}= m \ge 5$. 
Now if $S' \in \FM(S)$, by the same reason we get $S' \cong \PP(\mc{O}_{E'} \oplus \mc{L}')$ for some 
$\mc{L}'\in \Pic^0 E'$ with 
$$m= \lambda_{\pi} = \ord \mc{L}=\ord \mc{L}' .
$$ Moreover, by Corollary \ref{corollary:base-elliptic}, we see that $E \cong E'$. 

We divide the proof of Theorem \ref{thm:main} into two cases:  The case $m=p^e\ge 5$ for some $e>0$, and the case arbitrary $m\ge 5$ with $m\neq p^e$.
In both cases, first we define an injective map
%we shall show the existence of the inclusion
\begin{align}\label{eqn:inj}
\{J^i(S) \mid i \in (\Z/m\Z)^* \} / \cong 
 \hookrightarrow \{ \PP ( \mc O_{E}\oplus \mc L^i)   \mid i \in (\Z/m\Z)^* \}/\cong,
\end{align}
and secondly we shall see
\begin{equation}\label{eqn:Hpi_HE}
\lvert H_\pi \rvert \le \lvert H^{\cL}_{\hat{E}} \rvert.
\end{equation} 
The cardinality of the L.H.S in  \eqref{eqn:inj} is $\varphi(m)/\lvert H_\pi \rvert$ 
by Lemma \ref{lem:FM_group}, and
the cardinality of the R.H.S. in \eqref{eqn:inj}  is $\varphi(m)/\lvert H^{\cL}_{\hat{E}} \rvert$ by Lemma \ref{lem:projbdl}. Therefore, 
combining \eqref{eqn:inj} with \eqref{eqn:Hpi_HE}, we can conclude 
that \eqref{eqn:inj} is a bijection, and hence 
Theorem \ref{BMelliptic} yields
$$
\FM(S) =\{ \PP ( \mc O_{E}\oplus \mc L^i)   \mid i \in (\Z/m\Z)^* \}/\cong
$$ 
as required in  Theorem \ref{thm:main}.

\paragraph{Case:  $m=p^e\ge 5$ for some $e>0$.}
Theorem \ref{TUmain} implies that 
 $J^i(S)\cong \PP(\cO_E \oplus \cL_{i})$ for some $\cL_i\in \Pic^0{E}$ with $\ord\cL_i= p^e$. 
But in this case, $E$ is necessarily ordinary, and hence $\hat{E}[p^e]$ is a cyclic group generated by $\cL$.
So in this case, $\cL_i \cong \cL^{\beta(i)}$ for some $\beta(i) \in (\Z/m\Z)^*$, and thus we can define an injective map \eqref{eqn:inj} by $J^i(S) \mapsto \PP(\cO_E \oplus \cL^{\beta(i)})$. 

Denote by $F_0$ the elliptic curve satisfying $\hat{F_0} = \hat{E}/\Span{\mc{L}}$ as in \S \ref{subsec:i-2}. Then by \eqref{eqn:hi}, 
a general fiber of the elliptic fibration $\pi\colon S\to \PP^1$ is isomorphic to $F_0$.

\begin{claim}\label{claim:LHS}
The inequality \eqref{eqn:Hpi_HE} holds (if $m=p^e\ge 5$). 
\end{claim}

\begin{proof}
\cite[Propositions 5.3.3, 5.3.6]{MR986969} implies that $\kappa(J^0(S))=-\infty$. Combining this fact with \cite[Corollary 5.3.5]{MR986969}, we see that $J^0(S)$ is an elliptic ruled surface  with a section. Therefore, by the classification in Theorem \ref{TUmain} and  \cite[Theorem 5.3.1~(i)]{MR986969}, we have $J^0(S)\cong F_0\times \PP^1$. 
Then we have $\Br(J^0(S))=0$ by \cite[Proposition 2.1]{MR612710}. Moreover 
 we have 
$\lambda_\pi=p^e=\lambda_\pi'$ by \eqref{eqn:pipi'}, and hence the group
$H'_{\pi}$ in Lemma \ref{lem:H_1H_2} is trivial. Therefore Lemma \ref{lem:H_1H_2} yields 
$$\bigl| H_{\pi} \bigr|    \le    \bigl| \Aut_0(J^0_\eta)\bigr|.$$
Recall that   $H^{\cL}_{\hat{E}}=\Aut_0(E)$ by Lemma \ref{lem:injection} (ii) in the case $m=p^e\ge 5$. Hence,
to obtain the conclusion, it suffices to check that 
$\lvert\Aut_0(J^0_{\eta}) \rvert \le \lvert\Aut_0(E)\rvert$. 
Thus we may assume $2<|\Aut_0(J^0_{\eta})|$. 
Note that we have a surjective homomorphism
$$
\Aut_0(J^0(S)/\PP^1)\to \Aut_0(J^0_{\eta}),
$$ 
where $\Aut_0(J^0(S)/\PP^1)$ means the automorphism group of $J^0(S)(\cong F_0\times \PP^1)$ over $\PP^1$, fixing the $0$-section. 
Thus, we have an isomorphism
$\Aut_0(J^0(S)/\PP^1)\cong \Aut_0(F_0)$, and moreover obtain 
$$2< |\Aut_0(J^0_{\eta})| = |\Aut_0(J^0(S)/\PP^1)| = |\Aut_0(F_0)|.$$ 
This yields $j(F_0) = 0$ or $1728$. 
Since the morphism $q_1\colon F_0\to E$ obtained in \eqref{eqn:q1} is a composition of relative Frobenius morphisms (cf.~\cite[Theorem V.3.1]{Si09}),  
\cite[Exercise IV.4.20(a)]{Ha77} produces the isomorphism $E\cong F_0$, which completes the proof.
\end{proof}

Claim \ref{claim:LHS} completes the proof of Theorem \ref{thm:main}
in the case $m=p^e\ge 5$.

\paragraph{Case: Arbitrary $m\ge 5$ with $m\ne p^e$ for any $e>0$. }
We may put $m=np^e$ with $e\ge 0$, $n>1$,  $p \nmid n$. 
%(In the case $n=1$, we have already shown the result in the previous case.)
We generalize the method of \cite{Ue17} below.

Recall that $S\cong \PP(\mc{O}_E\oplus\mc{L})$, and 
define elliptic curves $F_0$ and $F$ as $\hat{F}_0:= \hat{E}/\Span{\mc{L}}$ and $\hat{F} := \hat{E}/\Span{\mc{L}^{p^e}}$. Denote by  
$$q_E\colon F\to E$$ the dual morphism of the quotient morphism
$
\hat{E} \to \hat{F} = \hat{E}/\Span{\mc{L}^{p^e}}.
$
Set $$\mc{M}:=q_E^*\mc{L} \ \text{ and } \
T:=\PP(\mc{O}_F\oplus\mc{M}).$$ 
Then  we see $\hat{F_0} = \hat{F}/\Span{\mc{M}}$ and  $\ord \mc{M}=p^e$. Moreover if $e>0$, the existence of a non-zero element $\mc{M}$ of $\hat{F}[p^e]$ implies that $F$ is ordinary, and the dual morphism of the quotient morphism
$$
\hat{F} \to \hat{F_0} = \hat{F}/\Span{\mc{M}}.
$$
is the $e$-th iteration of the relative Frobenius morphisms (cf.~\cite[Theorem V.3.1]{Si09}). 
Then we obtain the following commutative diagram:
\begin{equation}\label{eqn:F0FE}
\xymatrix{
F_0 \ar[d]_{\Fr^e} & F_0\times \PP^1 \ar[l]_{p_1} \ar[r]^{p_2} \ar[d]^{h_1} & \PP^1 \ar[d]^{\Fr_{\PP^1}^e} \\
F \ar[d]_{q_E}& T \ar[r]_{\pi_1} \ar[d]^{q} \ar[l]^{f_1} \ar@{}[lu]|{\Box} & \PP^1 \ar[d]^{q_{\PP^1}}\\
E & S \ar[r]_{\pi} \ar[l]^{f} \ar@{}[lu]|{\Box} & \PP^1\\
}
\end{equation}
Both of the left squares are fiber product diagrams, and 
the right squares are obtained by the Stein factorizations of
$\pi_1\circ h_1$ and $\pi\circ q$ respectively.
Moreover we have 
$$\deg q_E=\deg q=\deg q_{\PP^1}=n.$$

Take 
\begin{equation}\label{eqn:im}
i\in \Z\mbox{ with } 1\le i< m, \quad (i, m)=1.
\end{equation}
Note that this condition implies that $(i,p^e)=(i,n)=1$, and hence we sometimes regard $i\in (\Z/p^e\Z)^*$ or $i\in (\Z/n\Z)^*$ below.  

%By the previous result in the case $m=p^e$, we have an isomorphism 
Recall that we have already proved Theorem \ref{thm:main} for line bundles whose order is $p$-th power. By applying it to $\cM$, we obtain
\begin{equation}\label{eqn:JTPP}
J^i(T)\cong \PP(\mc{O}_F\oplus \mc{M}^{\beta(i)})
\end{equation}
for some $\beta (i)\in (\Z/p^e\Z)^*$.
Moreover,
since $(\Fr^e)^*\mc{M}\cong \mc{O}_{F_0}$, we have a diagram
\begin{equation}\label{eqn:Fre}
\xymatrix{
F_0 \ar[d]_{\Fr^e} & F_0\times \PP^1 \ar[l]_{p_1} \ar[r]^{p_2} \ar[d]^{h_i} & \PP^1 \ar[d]^{\Fr_{\PP^1}^e} \\
F & J^i(T) \ar[r]_{\pi_i} \ar[l]^{f_i} \ar@{}[lu]|{\Box} & \PP^1\\
}
\end{equation}
as in \eqref{eqn:pull_back}. Here $f_i$ is a $\PP^1$-bundle defined by using the $\PP^1$-bundle structure on $\PP(\mc{O}_F\oplus \mc{M}^{\beta(i)})$ and the isomorphism \eqref{eqn:JTPP}.

Fix an $n$-th primitive root of unity $\zeta$. Consider the multiplication on $\mathbb{G}_m$ by $\zeta$, 
and extend it  to the automorphism of $\PP^1$. Denote it by $g_{\PP^1}$.
Because we see that $q_{\PP^1}$ in \eqref{eqn:F0FE} fixes points $0$ and $\infty$ in $\PP^1$,  it turns out that the morphism $q_{\PP^1}$ is the quotient morphism by the action  of the group 
$\Span{g_{\PP^1}}\cong \Z/n\Z$ on $\PP^1$.

Take $a\in F$ such that $E\cong F/\Span{a}$ and $\ord a(=\ord \cL^{p^e})=n$.  Then we can construct an action of the group $G:=\Z/n\Z$ on $J^i(T)$ as follows.

\begin{claim}\label{cla:action}
For each $s\in (\Z/n\Z)^*$ and $t\in  (\Z/p^e\Z)^*$, 
there exists an automorphism $g_{s}$ of $J^{t}(T)$ which induces the translation $T_{s\cdot a}$ of $F$ and the automorphism $g_{\PP^1}$ of $\PP^1$.
\end{claim}

\begin{proof}
Since $T_{s\cdot a}^*\mc{M}\cong\mc{M}$, there exists an automorphism 
$$\alpha\in\Aut (J^t(T))(\stackrel{ \eqref{eqn:JTPP}}\cong \Aut (\PP(\mc{O}_F\oplus \mc{M}^{\beta(t)})))$$
 compatible with 
  $T_{s\cdot a}$ on $F$. Note that $T_{s\cdot a}$ lifts a translation $T_{s\cdot b}$ on $F_0$ for some $b\in F_0$ with $\Fr^e(b)=a$, and hence $\alpha$ lifts to $T_{s\cdot b}\times \id_{\PP^1}$ on $F_0\times \PP^1$.
\begin{equation*}%\label{eqn:beta}
\xymatrix@!0{
& F_0 \ar@{->}'[d][dd]\ar[dl]_{T_{s\cdot b}} & 
& F_0\times \PP^1 \ar[dl]%_<(.6){T_{s\cdot b}\times \id_{\PP^1}} 
\ar@{->}'[d][dd]\ar[ll]\ar[rr]&& \PP^1 \ar@{=}[dl]\ar[dd]\\
F_0\ar[dd]_{\Fr^e}& & F_0\times \PP^1 \ar[dd]^(.3){h_t}\ar[ll]
\ar[rr]&&\PP^1\ar[dd]&\\
& F\ar[dl]^<(0){T_{s\cdot a}} & & J^t(T)\ar@{->}'[l][ll]\ar@{->}'[r][rr] \ar[ld]^(.4)\alpha & &\PP^1\ar@{.>}[dl]^<(.4){\id_{\PP^1}} \\
F & & J^t(T) \ar[ll]^{f_t}\ar[rr]_{\pi_t} & &\PP^1& 
}
\end{equation*}
Therefore, $\alpha$ respects the elliptic fibration  $\pi_t$, i.e. $\alpha\in \Aut (J^t(T)/\PP^1)$. 

Next take an integer $q$ with $p^eq=1$ in $(\Z/n\Z)^*$. It follows from Lemma \ref{lem:autoPP} that
 there exists an automorphism $\beta\in\Aut(J^{t}(T)/F)$ which induces the automorphism $g_{\PP^1}^q$ on each fiber 
 $F_{f_t}$ (which we identify with $\PP^1$ by \eqref{eqn:PP}) of the $\PP^1$-bundle $f_t$. 
Combining \eqref{eqn:PP} with the commutativity of the right square in \eqref{eqn:Fre}, we  see that $\pi_t|_{F_{f_t}}\colon F_{f_t}\to \PP^1$ coincides with $\Fr^e_{\PP^1}$, 
 and then $\beta$ induces the automorphism 
$(g_{\PP^1})^{p^eq}=g_{\PP^1}$ on $\PP^1$, the base space of $\pi_t$.
\begin{equation*}\label{eqn:beta}
\xymatrix{
\PP^1 \ar[d]_{g_{\PP^1}^q}\ar[r]_\cong \ar@/^1.5pc/[rrr]|{\Fr^e_{\PP^1}}
&F_{f_t}\ar@{^{(}->}[r] &
J^t(T) \ar[r]_{\pi_t} \ar[d]^{\beta} & \PP^1 \ar[d]^{(g_{\PP^1})^{p^eq}=g_{\PP^1}} \\
\PP^1 \ar[r]^\cong \ar@/_1.5pc/[rrr]|{\Fr^e_{\PP^1}}
& F_{f_t}\ar@{^{(}->}[r]  \ar[r]  &J^t(T) \ar[r]^{\pi_t}  & \PP^1,\\
}
\end{equation*}
Hence, the automorphism $g_s:= \alpha\circ \beta$ has the desired property.
\end{proof}
Denote by $g$ a generator of the cyclic group $G=\Z/n\Z$, and define the action of $G$ on $J^{t}(T)$ by
\begin{equation}\label{eqn:action_st}
\rho_{s,t}\colon G\to \Aut (J^{t}(T)) \quad g\mapsto g_{s}.
\end{equation}
For the integer $i$ given in \eqref{eqn:im}, 
regard $i\in (\Z/n\Z)^*$ and $i \in  (\Z/p^e\Z)^*$, and
set $\rho_i:=\rho_{i,i}$.
We define the quotient variety to be
\begin{equation}\label{eqn:Si}
S_i:=J^i(T)/_{\rho_i} G
\end{equation}
by the action $\rho_i$,
and denote the quotient morphism by
$$
q_i \colon J^i(T)\to S_i.
$$ 
It is easy to see that $S$ is the quotient of $T=J^1(T)$ by the action $\rho_{s,1}$  for some $s$. Replace $a\in F$ with $s\cdot a$, and redefine $g_s$ and $\rho_{s,t}$ by this new $a$, so that $S=S_1$ holds. After this replacement,  we consider only the action $\rho_i$, but not general $\rho_{s,t}$. 

We set
$$
g^0_i:=T_{i\cdot b}\times g_{\PP^1}^q\in \Aut (F_0\times \PP^1).
$$
Then we see that $\ord  g^0_i=\ord T_{i\cdot b}=\ord  g_{\PP^1}^q =n$ and it is compatible with $g_i\in \Aut(J^i(T))$ defined in Claim \ref{cla:action}:
\begin{equation}\label{eqn:hrho}
h_i\circ g^0_i=g_i\circ h_i.
\end{equation}
We also define the action on $F_0\times \PP^1$ by 
\begin{equation}\label{eqn:action_0}
\rho^0_i\colon G\to \Aut (F_0\times \PP^1) \quad g\mapsto g^0_i
\end{equation}
for each $i$. 

Take an integer $j$ with $1\le j< m$, $(j, m)=1$ and $ij=1$ in $(\Z/m\Z)^*$. 
For the projection  
$$p_{13}\colon F_0  \times \Delta _{\PP ^1} \times F_0 \to F_0\times F_0,$$
define a line bundle
$$\mc U_0:=p_{13}^*\mc{O}_{F_0\times F_0}(\Delta_{F_0}+(j-1)F_0\times O+(i-1)O\times F_0)$$
on $$F_0  \times \Delta _{\PP ^1} \times F_0(\cong 
\noindent
(F_0\times \PP^1)\times_{\PP^1}(F_0\times \PP^1)).$$
Then $F_0\times \PP^1$ in the second factor in R.H.S. serves as $J^i(F_0\times \PP^1)$
where $\mc U_0$ plays the role of a universal sheaf, and moreover it is shown in \cite[page 3229]{Ue17}
 that it satisfies
\begin{equation}\label{eqn:g1gi0}
(\rho^0_1(g)\times \rho^0_i(g))^*\mc{U}_0\cong \mc{U}_0.
\end{equation} 

On the other hand, it follows from \cite[Theorem 5.3]{Br98} that we can  take a universal sheaf $\mc{U}'$ on $T\times_{\PP^1}J^i(T)$, which  satisfies that $\mc{U}'|_{z\times J^i(T)}$ is a line bundle of degree $j$ on $F_0$ for general $z\in T$.
For a point $(x,y)\in F_0\times (\PP^1\backslash \{ 0,\infty\})$,
there exists an isomorphism  
\begin{equation}\label{eqn:h1hi}
((h_1\times h_i)^*\mc{U}')|_{(F_0\times \PP^1)\times_{\PP^1} (x,y)}\cong 
\mc{U}'|_{T\times_{\PP^1} h_i((x,y))},
\end{equation}
since the restriction of $h_1\times h_i$ gives  
$$(F_0\times \PP^1)\times_{\PP^1} (x,y)\cong F_0\times y\cong  F_{y}\cong T\times_{\PP^1} h_i((x,y)),$$
where the second isomorphism comes from \eqref{eqn:hi}. 
Hence, we see that the L.H.S. in \eqref{eqn:h1hi}
 is a line bundle of degree $i$ on $F_0$. Then, by the universal property of $\mc{U}_0$,
there exists an automorphism $\phi_0\in \Aut (F_0)$ such that
$$(\id_{F_0  \times \Delta _{\PP ^1}}\times \phi_0)^*\mc{U}_0\cong (h_1\times h_i)^*\mc{U}'\otimes p_3^*\mc{N}_0$$ 
for some 
$\mc{N}_0\in \Pic^0 F_0$.

We shall construct an elliptic ruled surface $T'$ and (iso)morphisms $\phi_F, \phi, h'$
which make the following diagrams commutative:
\begin{equation}\label{eqn:fiber_product_0}
\xymatrix@!0{
& F_0 \ar@{->}'[d][dd]_{\Fr^e} \ar[dl]_{\phi_0} & 
& F_0\times \PP^1 \ar[dl]%_(.8){\phi_0\times \id}
\ar[dd]^{h_i}\ar[ll] \\
F_0\ar[dd]_{\Fr^e}& & F_0\times \PP^1 \ar[dd]^(.3){h'}\ar[ll] 
\\
& F\ar[dl]^<(0){\phi_F} & & J^i(T)\ar@{->}'[l][ll] \ar[ld]^\phi 
\\
F & & T' \ar[ll] 
}
\end{equation}
First, $\phi_0$ descends to $\phi_F\in \Aut (F)$ via $\Fr^e\colon F_0\to F$ by \cite[Corollary II.2.12]{Si09}, 
and $\phi_F$ induces an isomorphism 
$$\phi \colon   J^i(T)\cong \PP(\mc{O}_F\oplus \mc{M}^{\beta(i)})\to T':=\PP(\mc{O}_F\oplus \phi_{F*}\mc{M}^{\beta(i)}).$$ 
Note that 
$\phi_{F*}\in \Aut_0(\hat{F})$ 
preserves the subgroup 
$\ker \widehat{\Fr^e}=\hat{F}[p^e]=\Span{\mc{M}}$ 
of $\hat{F}$, and thus $\phi_{F*}\mc{M}^{\beta(i)}\in \Span{\mc{M}}$.
Hence we obtain a morphism 
$$
h'\colon F_0\times \PP^1\cong \PP(\mc{O}_{F_0}\oplus \mc{O}_{F_0}) \to T' \cong\PP(\mc{O}_F\oplus \phi_{F*}\mc{M}^{\beta(i)}),
$$
which fits into the diagram in \eqref{eqn:fiber_product_0}. 
%as $h_i$ in \eqref{eqn:Fre}. 
Moreover we
have the following commutative diagram:
\begin{equation*}
\xymatrix{
 F_0  \times \Delta _{\PP ^1}\times F_0  \ar[d]_{h_1\times h'} &  F_0  \times \Delta _{\PP ^1}\times F_0 \ar[l]_{(\id_{F_0  \times \Delta _{\PP ^1})}\times \phi_0} \ar[d]^{h_1\times h_i} \ar[r]^{\qquad p_3}& F_0 \ar[d]^{\Fr^e}\\
 T\times_{\PP^1}T' & T\times_{\PP^1}J^i(T) \ar[l]^{\id_T\times \phi}\ar[r]_{\qquad f_i\circ p_2}&F \\
}
\end{equation*}
Take $\mc{N}\in \Pic ^0F$ such that $(\Fr^{e})^*\mc{N}=\mc{N}_0$, and  
define a line bundle
$$\mc{U}:=(\id_T\times \phi)_*(\mc{U}'\otimes (f_i\circ p_2)^*\mc{N})
$$ 
on $T\times_{\PP^1}T'$
so that 
\begin{equation}\label{eqn:U_0U}
\mc{U}_0\cong (h_1\times h')^*\mc{U}
\end{equation}
holds. The pair $(T',\mc{U})$ serves as $J^i(T)$ and its universal sheaf, ane thus we redefine $T'$ to be $J^i(T)$.

\begin{claim}\label{eqn:g1gi}
The universal sheaf $\mc{U}$ on $T\times_{\PP^1}J^i(T)$  satisfies 
$$
(\rho_1(g)\times \rho_i(g))^*\mc{U}\cong \mc{U}.
$$
\end{claim}
\begin{proof}
Take $y_0\in \PP^1\backslash \{0,\infty\}$ with $y:=\Fr^e(y_0)\in \PP^1\backslash \{0,\infty\}$.
Denote by $F_y\times F'_y$ the fiber of 
$
\pi_1\times \pi_i \colon T\times_{\PP^1}J^i(T)\to \PP^1
$
over the point $y$. 
Pull back the isomorphism \eqref{eqn:U_0U} to the subscheme $F_0\times y_0\times  F_0$, which is isomorphic to $F_y\times F'_y$ by \eqref{eqn:hi}, and combine \eqref{eqn:hrho} and \eqref{eqn:g1gi0} with it, then we have isomorphisms 
$$
((\rho_1(g)\times \rho_i(g))^*\mc{U})|_{F_y\times F_y}\cong  ((\rho^0_{1}(g)\times \rho^0_{i}(g))^* \mc{U}_0)|_{F_0\times y_0\times F_0}\cong\mc{U}_0|_{F_0\times y_0\times F_0}\cong \mc{U}|_{F_y\times F_y}.
$$
\begin{equation*}
\xymatrix{
F_0\times y_0\times F_0 \ar@{^{(}->}[r]\ar[d]_\cong & F_0  \times \Delta _{\PP ^1}\times F_0  \ar[d]_{h_1\times h_i}\ar[r]^{\qquad p_2} & \PP ^1 \ni y_0 \ar@<-2ex>[d]^{\Fr^e_{\PP^1}} \\
F_y\times F'_y \ar@{^{(}->}[r]  &T\times_{\PP^1}J^i(T) \ar[r]_{\quad \pi_1\times \pi_i} & \PP^1 \ni y \\
}
\end{equation*}
This yields that the line bundle $L:=(\rho_1(g)\times \rho_i(g))^*\mc{U}\otimes \mc{U}^{-1}$ is trivial over the open set $(\pi_1\times \pi_i)^{-1}(\PP^1\backslash \{0,\infty\})$ by  \cite[Exercise III.12.4]{Ha77}.  
We also see by \eqref{eqn:hrho}, \eqref{eqn:g1gi0} and \eqref{eqn:U_0U} that $(h_1\times h_i)^*L$ is trivial over $\PP^1\backslash \{0,\infty\}$, and thus 
\begin{equation}\label{eqn:D0D0}
L\cong \mc{O}_{T\times_{\PP^1}J^i(T)}(b(D_0\times D'_0-D_\infty\times D'_\infty))
\end{equation}
for some $b\in \Z$, where $p^eD_0$ and $p^eD'_0$ (resp.~$p^eD_\infty$ and $p^eD'_\infty$) are the multiple fibers over $0\in \PP^1$ (resp.~$\infty$) of $\pi_1$ and $\pi_i$. 
Note that $\ord L$ divides $p^e$, the multiplicity of the multiple fibers.
Since $\ord (\rho_1(g)\times \rho_i(g))=n$ and the R.H.S. in \eqref{eqn:D0D0} is $(\rho_1(g)\times \rho_i(g))$-invariant,
we see that 
$$
\mc{U}\cong (\rho_1(g)\times \rho_i(g))^{n*}\mc{U}\cong(\rho_1(g)\times \rho_i(g))^{(n-1)*}\mc{U}\otimes L\cong\cdots \cong \mc{U}\otimes L^{\otimes n},
$$
and hence $\ord L \mid n$. Since  $p \nmid n$, 
we have $\ord L=1$, as it is required.  
\end{proof}

Recall that we have the following commutative diagram by the definition of $S_i$ in \eqref{eqn:Si}:
\begin{equation*}%\label{eqn:Si}
\xymatrix{
F \ar[d]_{q_E} & J^i(T) \ar[l]_{f_i} \ar[r]^{\pi_i} \ar[d]^{q_i} & \PP^1 \ar[d]^{q_{\PP^1}}  \\
E & S_i \ar[r]_{\pi_{S_i}}\ar[l]\ar@{}[lu]|{\Box} & \PP^1\\
}
\end{equation*}
Here, $q_E$ and $q_{\PP^1}$ are the same one appeared in \eqref{eqn:F0FE},
and $\pi_{S_i}$ is an elliptic fibration.

\begin{claim}\label{cla:case2}
For each $i$, there exists $\alpha(i) \in (\Z/m\Z)^*$ such that we have an isomorphism 
$$S_i\cong \PP(\cO_E \oplus \cL^{\alpha(i)}).$$ 
over $E$.
\end{claim}

\begin{proof}
First of all, we know by Theorem \ref{TUmain} that there exisits an isomorphism  
 $S_i\cong \PP(\cO_E \oplus \cL_{i})$ 
 over $E$ for some $\cL_i\in \Pic^0{E}$ with $\ord \cL_i=m$. 
Then the result follows from 
$$
\cL_i\in \ker (\widehat{\Fr^e}\circ \widehat{q_E})=\Span{\cL}\cong \Z/m\Z.
$$ 
\end{proof}

Recall that $S=S_1$ below.

\begin{claim}\label{cla:JS}
There exists an isomorphism $J^i(S)\cong S_i$.
\end{claim}

\begin{proof}
First, we shall show that there exists a coherent sheaf $\mc{U}_i$ on $S\times S_i$ such that
\begin{equation}\label{eqn:descent}
 (q_{1}\times {\id}_{J^i(T)})_*\mc{U}\cong  ( {\id}_{S}\times q_i)^*\mc{U}_i
\end{equation}
for the morphisms
$$
T\times J^i(T) \stackrel{q_1\times\id_{J^i(T)}}\to S\times J^i(T)\stackrel{\id_S\times q_i}\to S\times S_i.
$$
Claim \ref{eqn:g1gi} implies that
$$
(\rho_1(g)\times  {\id}_{J^i(T)})^*\mc U\cong ( {\id}_{T}\times \rho_i(g)^{-1})^*\mc U.
$$
Push forward the both sides by the morphism $q_{1}\times {\id}_{J^i(T)}$. Then we obtain 
$$
(q_{1}\times {\id}_{J^i(T)})_*\mc{U}\cong ({\id}_{S}\times \rho_i(g)^{-1})^* 
(q_{1}\times {\id}_{J^i(T)})_*\mc{U},
$$
that is, the sheaf
$
 (q_{1}\times {\id}_{J^i(T)})_*\mc{U}
$
is $G$-invariant with respect to the diagonal action of $G$ on $S\times J^i(T)$, 
where $G$ acts on $S$ trivially. 
Since $G=\Span{g}$ is a finite cyclic group, the $G$-invariance of coherent sheaves is equivalent to the $G$-equivariance, and hence there exists a coherent sheaf $\mc{U}_i$ on $S\times S_i$ 
satisfying \eqref{eqn:descent}.

For $z\in J^i(T)$, we have
$$
\mc{U}_i|_{S\times q_i(z)}\cong ((q_{1}\times {\id}_{J^i(T)})_*\mc{U})|_{S\times z}
\cong q_{1*} (\mc{U}|_{T\times z}).
$$
Here, the second isomorphism follows from \cite[Lemma 1.3]{BO95} and the smoothness of $q_1$.
Suppose that $z$ is not contained in multiple fibers of $\pi_i$, that is, $y:=\pi_i(z)\in \PP^1\backslash \{0,\infty\}$ by the convention stated in \S \ref{subsec:i-2}.
Then $\mc{U}|_{T\times z}$ is actually a sheaf on $F_y\times z$, and 
 the restriction $q_1|_{F_y\times z}$  is an isomorphism by \eqref{eqn:hi}.
It turns out that $\mc{U}_i|_{S\times {q_i(z)}}$ is 
also a line bundle of degree $i$ on $F_{q_{\PP ^1}(y)}\times q_i(z)$.
%Here, $F_y(\cong F_0)$ is a fiber of $\pi$ over the point $q_{\PP ^1}(y)$. 

Then, by the universal property of $J^i(S)$,
there exists a morphism from 
$$\pi_{S_i}^{-1}(\PP^1\backslash\{0,\infty\})(\subset S_i)\to \pi_{J^i(S)}^{-1}(\PP^1\backslash\{0,\infty\})(\subset J^i(S))$$
  over $\PP^1\backslash\{0,\infty\}$, where $\pi_{S_i}$ and $\pi_{J^i(S)}$ are the elliptic fibrations on $S_i$ and $J^i(S)$ respectively.
Since $\mc{U}_i|_{S\times q_i(z_1)}\not\cong\mc{U}_i|_{S\times q_i(z_2)}$ on $F_y$ for $z_1\ne z_2\in J^i(T)$, this morphism is injective, and hence  
$S_i$ and $J^i(S)$ are birational over $\PP^1$.
Then, \cite[Proposition III.8.4]{BHPV} implies that $S_i \cong J^i(S)$. 
\end{proof}

Combining Claims \ref{cla:case2} and \ref{cla:JS},
we obtain the inclusion \eqref{eqn:inj} by the map
$$
J^i(S)\mapsto  \PP(\cO_E \oplus \cL^{\alpha(i)}).
$$
The next aim is to show \eqref{eqn:Hpi_HE}.

\begin{claim}\label{claim:alpha}
There exists an injective group homomorphism 
$$
\overline{\alpha} \colon H_\pi/\{\pm 1\}\to H^\mc{L}_{\hat{E}}/\{\pm 1\}.
$$
%such as $J^i(S)\cong \PP(\mc{O}_E\oplus \mc{L}^{\alpha(i)})$ for $i\in H_\pi$.
\end{claim}

\begin{proof}
Take $i \in H_\pi (:=\{ i\in (\Z/m\Z)^* \mid J^i(S)\cong S\}).$ 
We have $\alpha(i)\in (\Z/m\Z)^*$ so that there exists an isomorphism
$$
\psi\colon \PP(\mc{O}_E\oplus \mc{L}^{\alpha(i)})\stackrel{\cong}\to S_i\stackrel{\cong}\to J^i(S)
$$ 
by Claims \ref{cla:case2} and \ref{cla:JS}. We use $\psi$ and the $\PP^1$-bundle structure
on $\PP(\mc{O}_E\oplus \mc{L}^{\alpha(i)})$ to fix a $\PP^1$-bundle structure on $J^i(S)$:
$$
f_{J^i(S)}\colon J^i(S)\to E
$$
Then Lemma \ref{lem:automor} (iii) implies that there exist  
an isomorphism $\varphi$ and an automorphism $\varphi_E\in \Aut_0(E)$ fitting in the commutative diagram
\begin{equation}\label{eqn:comm}
\xymatrix{
 \PP(\mc{O}_E\oplus \mc{L}^{\alpha(i)})\ar[r]^{\quad\psi}\ar[d] & J^i(S) \ar[d]_{f_{J^i(S)}}  \ar[r]^{\varphi}& S\ar[d]^{f}\\ 
E \ar@{=}[r]&E \ar[r]_{\varphi_E} & E}   \\
\end{equation}
and
$\varphi_E^*\mc{L}\cong  \mc{L}^{\alpha(i)}$ is satisfied.

Take another isomorphism $\varphi'\colon  J^i(S)\to S.$
Then since $\varphi'\circ \varphi^{-1}$ is an automorphism of $\PP(\mc{O}_E \oplus\mc{L})$, we have 
$(\varphi'_E\circ \varphi_E^{-1})^*\mc{L}\cong \mc{L}^{\pm 1}$
by Lemma \ref{lem:automor} (i) and (ii). Thus we obtain the group homomorphism
$$
\alpha \colon H_\pi\to H^\mc{L}_{\hat{E}}/\{\pm 1\}(:= \{ i \in (\Z/m\Z)^* \, | \, \exists \phi \in \Aut_0(E) \, \text{s.t.} \, \phi^*\cL \cong \cL^{i}\}/\{\pm 1\}.)
$$
Thus it suffices to prove $\Ker \alpha=\{\pm 1\}$. 
Suppose $i\in \Ker \alpha$. 
Since $\varphi_E^*\mc{L}\cong  \mc{L}^{\pm 1}$ holds in this case, Lemma \ref{lem:action}
implies that $\varphi_E$ fitting in the diagram \eqref{eqn:comm} is either $\id_E$ or $-\id_E$.
Replace $\varphi$ with $f^*(-\id_E)\circ \varphi$ (see the notation in  Lemma \ref{lem:automor} (ii) and the proof of ibid.~(iii)) if necessary, then we may assume that $\varphi_E=\id_E$.
We have the following commutative diagram \footnote{Here, we identify $S_i$ and $J^i(S)$ by Claim \ref{cla:JS}.}:
\begin{equation}\label{eqn:fiber_product}
\xymatrix@!0{
& F\ar@{->}'[d][dd]\ar@{=}[dl]
& & J^i(T)\ar[ll]_{f_i} \ar[dd]\ar@{.>}[dl]^<(.6){\exists\phi}
\\
F \ar[dd]_{q_E}
& & T \ar[ll]_{\quad f_1}\ar[dd]
\\
& E \ar@{=}[dl]
& & S_i\ar[dl]^\varphi\ar@{->}'[l][ll]
\\
E 
& & S \ar[ll]
}
\end{equation}
Because the front and the back squares in \eqref{eqn:fiber_product} are the fiber product diagrams,
there exists an isomorphism $\phi\colon J^i(T)\to T$ which makes the right square  the fiber product. 

Since $\phi$ descends to $\varphi\colon  S_i=J^i(T)/_{\rho_i}G\to S=T/_{\rho_1}G$ for $G=\Z/n\Z=\Span{g}$,  we have 
$$
\rho_1(g)\circ \phi=\phi\circ \rho_i(g)^l
$$
for some $l$. Recall that both of $\rho_1(g)$ and $\rho_i(g)$ induce the same automorphism $g_{\PP^1}$ on the base curve $\PP^1$ of the elliptic fibrations on $T$ and $J^i(T)$ (see Claim \ref{cla:action}  and \eqref{eqn:action_st}), then we see  
$l=\pm 1$.
Next recall $\rho_1(g)$ (resp.~$\rho_i(g)$) induces the  automorphism $T_a$ (resp.~$T_{i\cdot a}$) on $F$, the base curve of the $\PP^1$-bundle $f_1$ (resp.~$f_i$). Then we know that 
$$
T_a=(T_{i\cdot a})^l=T_{li\cdot a},
$$
and hence, $1=il$ in $(\Z/n\Z)^*$. Therefore we have $i=\pm 1$, and hence
$\Ker \alpha\subset \{\pm 1\}$. The other direction is obvious.
\end{proof}

By Claim \ref{claim:alpha}, 
we conclude that  $|H_\pi |\le  |H^\mc{L}_E|$ as is required in \eqref{eqn:Hpi_HE}.

Therefore, we complete the proof of the first statement in Theorem \ref{thm:main} for arbitrary $m \ge 5$. 
The second follows from Lemma \ref{lem:injection} (ii).

\bibliographystyle{plain}
%\bibliography{ref.bib}

\noindent
Hokuto Uehara

Department of Mathematical Sciences,
Graduate School of Science,
Tokyo Metropolitan University,
1-1 Minamiohsawa,
Hachioji,
Tokyo,
192-0397,
Japan 

{\em e-mail address}\ : \  hokuto@tmu.ac.jp
\ \vspace{0mm} \\

\noindent
Tomonobu Watanabe

Hosoda Gakuen Junior and Senior High School, 2-7-1, Honcho, Shiki, Saitama, 353-0004, Japan

{\em e-mail address}\ : \  ttkk.8128@gmail.com
\ \vspace{0mm} \\

\end{document}